\documentclass{article}
\usepackage{graphicx} % Required for inserting images
\usepackage[margin=1in]{geometry}
\usepackage{amsmath}
\usepackage{amsthm}
\usepackage{amssymb}
\usepackage{mathtools}
\usepackage{thmtools}
\usepackage{xcolor}
\usepackage{float}
\usepackage[shortlabels]{enumitem}
\setlist[itemize]{topsep=2ex,itemsep=0pt}
\setlist[enumerate]{topsep=2ex,itemsep=0pt}
\usepackage{caption}
\usepackage{hyperref}
\usepackage{microtype,lmodern}
\usepackage[noabbrev,capitalise]{cleveref}
\crefname{lemma}{Lemma}{Lemmas}
\crefname{theorem}{Theorem}{Theorems}
\crefname{corollary}{Corollary}{Corollaries}
\crefname{proposition}{Proposition}{Propositions}
\crefname{conjecture}{Conjecture}{Conjectures}
\crefname{claim}{Claim}{Claims}
\crefname{equation}{}{}

\usepackage[english]{babel}
\usepackage[longnamesfirst,numbers,sort&compress]{natbib}
\makeatletter
\def\NAT@spacechar{~}
\makeatother
\usepackage{csquotes}

\renewcommand{\thefootnote}{\fnsymbol{footnote}}
\usepackage[bottom]{footmisc}

\declaretheorem[numberwithin=section]{theorem}
\declaretheorem[numbered=no, name=Theorem]{theorem*}
\declaretheorem[
    numberwithin=section, 
    name=Claim
]{claim}
\declaretheorem[numberlike=theorem]{lemma,proposition,problem,corollary,conjecture}

\DeclarePairedDelimiter{\abs}{|}{|}

\DeclareMathOperator{\had}{had}

\newcommand{\Cc}{\mathcal{C}}

\newcommand{\triqed}{\hfill\ensuremath{\vartriangleleft}}

\title{\textbf{Hadwiger's Conjecture for $\{\text{co-claw}, \text{co-gem}\}$-free graphs\\ and $\{\text{fork}, \text{antifork}\}$-free graphs}}
\author{Daniel Carter\footnotemark[1]\and Jung Hon Yip\footnotemark[2]}
\date{\today}

\usepackage{pgfplots}
\usetikzlibrary{arrows.meta}
\usepackage{tikz, graphicx, tikz-cd}
\usepackage{float}

\tikzset{
    dot/.style = {circle, fill, minimum size=#1,
            inner sep=0pt, outer sep=0pt},
    dot/.default = 4pt
}

\begin{document}

\footnotetext[1]{Princeton University, Princeton, USA. \texttt{dc65@princeton.edu}.}

\footnotetext[2]{Monash University, Melbourne, Australia. Supported by a Monash Graduate Scholarship. \texttt{junghon.yip@monash.edu}.}

\maketitle

\begin{abstract}
We prove Hadwiger's Conjecture for $\{\text{co-claw}, \text{co-gem}\}$-free graphs and $\{\text{fork}, \text{antifork}\}$-free graphs, where the \textit{co-claw} is the disjoint union of a triangle and a vertex, the \textit{co-gem} is the disjoint union of a 4-vertex path and a vertex, the \textit{fork} is obtained from $K_{1,3}$ by subdividing one of the edges, and the \textit{antifork} is the complement of the fork.  The $\{\text{co-claw}, \text{co-gem}\}$-free graphs include the complements of line graphs of triangle-free multigraphs, and thus our results imply Hadwiger's Conjecture for these graphs. In fact, we prove a stronger result: every $\{\text{co-claw}, \text{co-gem}\}$-free graph $G$ has a $K_{\chi(G)}$-model where each branch set has size at most 2, and every $\{\text{fork}, \text{antifork}\}$-free graph $G$ has a $K_{\chi(G)}$-model where at most one branch set has size greater than 2.
\end{abstract}

\renewcommand{\thefootnote}
{\arabic{footnote}}

\section{Introduction}
Graphs in this paper are finite, loopless and with no parallel edges; graphs allowing parallel edges will be called \textit{multigraphs}. A graph $H$ is a \textit{minor} of a graph $G$ if $H$ can be obtained from $G$ by deleting vertices and edges, and by contracting edges. 
The maximum integer $n$ such that the complete graph $K_n$ is a minor of $G$ is called the \textit{Hadwiger number} of $G$, denoted $\had(G)$. Recall that $\omega(G)$ is the size of the largest clique in $G$, and $\chi(G)$ is the chromatic number of $G$. We use the notation $A\sqcup B$ to denote the disjoint union of two graphs $A$ and $B$, and $\overline{G}$ for the complement of $G$. For a graph $H$, we say that $G$ is \textit{$H$-free} if no induced subgraph of $G$ is isomorphic to $H$. Likewise for a set $S$ of graphs, $G$ is \textit{$S$-free} if $G$ is $H$-free for all $H \in S$. A class of graphs is \textit{hereditary} if it is closed under taking induced subgraphs. For any undefined graph theory notation, see Diestel~\cite{Diestel05}.

One of the deepest unsolved problems in graph theory is Hadwiger's Conjecture:
\begin{conjecture}[\cite{Hadwiger43}]
For all graphs $G$, $\had(G)\ge \chi(G)$.
\end{conjecture}
See~\cite{Kawa15,Toft96,SeymourHC} for surveys on Hadwiger's Conjecture in general. Recently, there has been interest in proving Hadwiger's Conjecture for various hereditary classes; see~\cite{CV20,YipHad26} for surveys and~\cite{Carter22,PST03,reed_hadwigers_2004,SongBrianHadwigerHole17} for other results on Hadwiger's Conjecture in hereditary classes.

The \textit{claw} is the complete bipartite graph $K_{1,3}$, and the \textit{co-claw} is the complement of the claw, which is $K_3\sqcup K_1$. The \textit{gem}, also known as the 4-fan, is the graph on five vertices formed by adding a vertex adjacent to all of the vertices in a 4-vertex path; the \textit{co-gem} is the complement of the gem, which is $P_4\sqcup K_1$. The \textit{fork} is obtained from the claw by subdividing one of the edges, and the \textit{antifork} is the complement of the fork (see \cref{fig:graphs0}).
In this paper, we prove Hadwiger's Conjecture for $\{\text{co-claw}, \text{co-gem}\}$-free graphs and $\{\text{fork}, \text{antifork}\}$-free graphs. Proving Hadwiger's Conjecture for $\{\text{fork}, \text{antifork}\}$-free graphs was a problem proposed by the first author at the 2025 Barbados Graph Theory Workshop~\cite{barbados}.

\begin{figure}[htb!]
    \centering
    \includegraphics{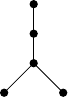}
    \qquad \quad
    \includegraphics{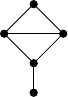}
    \qquad \quad
    \includegraphics{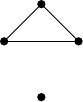}
        \qquad \quad
    \includegraphics{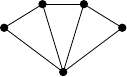}
    \qquad \quad
    \includegraphics{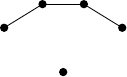}
    \caption{Left to right: the fork, antifork, co-claw, gem, and co-gem.} 
    \label{fig:graphs0}
\end{figure}

\begin{theorem}\label{thm:fork}
    If $G$ is $\{\text{fork}, \text{antifork}\}$-free, then $\had(G) \geq \chi(G)$.
\end{theorem}

\begin{theorem}\label{thm:gem}
    If $G$ is $\{\text{co-claw}, \text{co-gem}\}$-free, then $\had(G) \geq \chi(G)$.
\end{theorem}
We now motivate \cref{thm:fork,thm:gem}. Since Hadwiger's Conjecture has been proven for line graphs of multigraphs \cite{reed_hadwigers_2004,CO08}, it is natural to ask if it also holds for complements of line graphs. As a first step, our results imply Hadwiger's Conjecture for complements of line graphs of triangle-free multigraphs (see \cref{cor:multigraphs}).

Hadwiger's Conjecture remains difficult for hereditary classes characterized by a single forbidden induced subgraph. For example, despite having a structure theorem for claw-free graphs \citep{CP08ClawFree}, Hadwiger's Conjecture remains open for these graphs. In fact, even the special case of Hadwiger's Conjecture for $\overline{K_3}$-free graphs (these are graphs with independence number at most $2$, and every $\overline{K_3}$-free graph is claw-free) is open despite significant attention \cite{SeymourHC,CP08ClawFree,PST03,Blasiak07,CS12,NS26a,YipHad26}.\footnote{Seymour \cite{SeymourHC} writes of this case: ``This seems to me to be an excellent place to look for a counterexample. My own belief is, if it is
true for graphs with stability number two then it is probably true in general.''}
To date, the best partial result by Chudnovsky and Fradkin~\cite{chudnovsky_approximate_2010} is that $\had(G)\ge (2/3)\chi(G)$ if $G$ is claw-free. 
As another example, Hadwiger's Conjecture remains open for triangle-free graphs \cite{barbados17,DelcourtPostle24,DvorakYepremyanTriangleFree19}. There are some encouraging signs for triangle-free graphs: a breakthrough result of Delcourt and Postle \cite{DelcourtPostle24} shows that there is a constant $c>0$ such that every triangle-free graph $G$ satisfies $\had(G)\geq c\cdot\chi(G)$,
improving on previous work by Dvo{\v{r}}{\'a}k and Yepremyan \cite{DvorakYepremyanTriangleFree19}. This demonstrates the difficulty of proving Hadwiger's Conjecture in classes characterized by a single forbidden induced subgraph. 

However, what if we forbid two induced subgraphs?
In this case there are several positive results in the literature \cite{Carter22,PST03,reed_hadwigers_2004,SongBrianHadwigerHole17}. We mention a few results about fork-free graphs, which are a natural generalization of claw-free graphs.
Hadwiger's Conjecture holds for $\{\text{fork}, K_5 - e\}$-free graphs and $\{\text{fork}, \text{HVN}\}$-free graphs (see \cref{fig:graphs1}).
In these two classes, every graph $G$ satisfies $\chi(G) \le \omega(G) + 1$ \cite{randerath1998vizing}, and more generally, Hadwiger's Conjecture holds for hereditary classes $\Cc$ in which $\chi(G) \le \omega(G) + 1$ for all $G \in \Cc$~\cite{CV20}. However, there are $\{\text{fork}, \text{antifork}\}$-free graphs with $\chi(G) > \omega(G) + 1$. Thus, other techniques are required to prove Hadwiger's Conjecture for $\{\text{fork}, \text{antifork}\}$-free graphs.
We remark that any strengthening of \cref{thm:fork,thm:gem} by dropping one forbidden induced subgraph would imply a known hard case of Hadwiger's Conjecture. 
For example, Hadwiger's Conjecture for fork-free or co-gem-free graphs would imply it for $\overline{K_3}$-free graphs, while Hadwiger's Conjecture for antifork-free or co-claw-free graphs would imply it for triangle-free graphs.

Our primary tool for proving \cref{thm:fork} is the structure theorem for $\{\text{fork}, \text{antifork}\}$-free graphs by Chudnovsky, Cook, and Seymour (\cref{thm:structure}), which states that every $\{\text{fork}, \text{antifork}\}$-free graph satisfies one of eight outcomes.
It is straightforward to see that a minimal $\{\text{fork}, \text{antifork}\}$-free counterexample to Hadwiger's Conjecture cannot satisfy seven of the eight outcomes; we prove this in \cref{sec:fork}. 
The eighth outcome, when $G$ is the complement of a line graph of a simple triangle-free graph, is handled in \cref{sec:gem}. In fact, we prove a strengthening. We prove Hadwiger's Conjecture for $\{\text{co-claw}, \text{co-gem}\}$-free graphs (\cref{thm:gem}), which implies Hadwiger's Conjecture for the complements of line graphs of triangle-free multigraphs. Our proof uses the Strong Perfect Graph Theorem \cite{spgt}.

\begin{figure}[htb!]
    \centering
    \includegraphics{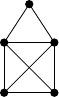}
    \qquad \quad
    \includegraphics{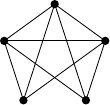}
    \qquad \quad 
    \includegraphics{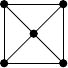}
    \qquad \quad 
    \includegraphics{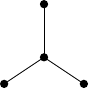}
    \caption{Left to right: the HVN, $K_5 - e$, $W_4$, and the claw.}
    \label{fig:graphs1}
\end{figure}

A \textit{minor model} of a graph $H$ in a graph $G$ (or \textit{$H$-model} for short) is a collection $M=\{S_x: x \in V(H)\}$ of pairwise disjoint non-empty connected subsets of $V(G)$ such that there is an edge between a vertex of $S_x$ and a vertex of $S_y$ whenever $xy\in E(H)$. The sets $S_x$ are known as the \textit{branch sets} of $M$. Note that $H$ is a minor of $G$ if and only if $G$ contains a $H$-model.
An $H$-model is \textit{small} if every branch set has size 1 or 2, and \textit{semi-small} if at most one branch set has size greater than $2$. Define $\had_2(G)$ to be the maximum integer $t$ such that $G$ has a small $K_t$-model and $\had_2^+(G)$ to be the maximum integer $t$ such that $G$ has a semi-small $K_t$-model.
It is straightforward to see that $\had(G)\ge \had_2^+(G)\ge \had_2(G)\ge \omega(G)$. We in fact prove a strengthening of \cref{thm:gem,thm:fork}:  
\begin{theorem}\label{thm:strong_fork}
If $G$ is $\{\text{fork}, \text{antifork}\}$-free, then $\had_2^+(G) \geq \chi(G)$.
\end{theorem}
\begin{theorem}\label{thm:strong_gem}
If $G$ is $\{\text{co-claw}, \text{co-gem}\}$-free, then $\had_2(G) \geq \chi(G)$.
\end{theorem}
Our results imply Hadwiger's Conjecture for complements of line graphs of triangle-free multigraphs, as we now explain.
The wheel $W_4$ is formed by adding a vertex adjacent to every vertex of $C_4$.
\begin{proposition}\label{prop:forbidden2}
$G$ is the line graph of a triangle-free multigraph if and only if it is $\{\text{claw}, \text{gem}, W_4\}$-free.
\end{proposition}
\begin{proof}
    This follows from Corollary 3.2(i) and Theorem 4.2 of~\cite{lgtf}.
\end{proof}
\cref{thm:strong_gem,prop:forbidden2} thus imply:
\begin{corollary}\label{cor:multigraphs}
    If $G$ is the complement of the line graph of a triangle-free multigraph, then $\had_2(G) \geq \chi(G)$.
\end{corollary}
\textbf{Outline of the paper.} In \cref{sec:gem}, we prove \cref{thm:strong_gem} and hence \cref{cor:multigraphs}.  \cref{thm:strong_fork} relies on \cref{thm:strong_gem} as well as the above-mentioned structure theorem of $\{\text{fork}, \text{antifork}\}$-free graphs (\cref{thm:structure}). In \cref{sec:fork}, we prove \cref{thm:strong_fork} by showing that a minimal $\{\text{fork}, \text{antifork}\}$-free counterexample to \cref{thm:strong_fork} fails all cases of the structure theorem. We conclude with open problems in \cref{sec:op}.

\section{Proof of Theorem~\ref{thm:strong_gem}}
\label{sec:gem}
In this section we prove \cref{thm:strong_gem}. A \textit{hole} is a cycle of length at least 4; an \textit{antihole} is the complement of a hole. A hole or antihole is \textit{odd} if it has an odd number of vertices. A graph $G$ is \textit{perfect} if $\omega(H)=\chi(H)$ for all induced subgraphs $H$ of $G$ (including $G$ itself). The celebrated Strong Perfect Graph Theorem states that the perfect graphs are precisely the $\{\text{odd hole}, \text{odd antihole}\}$-free graphs~\cite{spgt}.

A \textit{minimal counterexample} to \cref{thm:strong_gem} is a $\{\text{co-claw}, \text{co-gem}\}$-free graph $G$ such that $\had_2(G)<\chi(G)$ but $\had_2(H)\ge \chi(H)$ for all proper induced subgraphs $H$ of $G$. We prove \cref{thm:strong_gem} by showing that any minimal counterexample must be odd hole and odd antihole-free, and therefore perfect.

\begin{proof}[Proof of \cref{thm:strong_gem}]
Suppose $G$ is a minimal counterexample to \cref{thm:strong_gem}. First note that $G$ is $C_{2k+1}$-free for $k\ge 3$, since $G$ is $(P_4\sqcup K_1)$-free. We will prove that $G$ is also $\overline{C_{2k+1}}$-free for $k\ge 2$ by induction on $k$. Then $G$ is $\{\text{odd hole}, \text{odd antihole}\}$-free, so $G$ is perfect and $\had_2(G)\ge \omega(G)=\chi(G)$.

For the base case $k=2$, we prove that $G$ is $\overline{C_5}$-free, which is the same as being $C_5$-free. Suppose $X=\{x_1,\dots,x_5\}$ induces $C_5$ in $G$, so $x_i$ is adjacent to $x_j$ if and only if $i-j\equiv \pm 1\pmod 5$. In what follows, we take indices modulo 5. 

\begin{claim}
\label{claim:nghr5}
For each $v\in V(G)\setminus X$, one of the following holds:
\begin{enumerate}[label=(\theclaim\alph*), itemindent=1em]
    \item \label{u1} $N_G(v)\cap X=\{x_i,x_{i+2}\}\text{ for some }i\in\{1,\dots,5\}$, or
    \item \label{u2} $N_G(v)\cap X=\{x_i,x_{i+1},x_{i+3}\}\text{ for some }i\in\{1,\dots,5\}$ or
    \item \label{u3} $N_G(v)\cap X=X$.
\end{enumerate}
\end{claim}
If $N_G(v)\cap X=\varnothing$ or $N_G(v)\cap X=\{x_1\}$, then $G[x_2,x_3,x_4,x_5,v]\cong P_4\sqcup K_1$. Thus $|N_G(v) \cap X| \geq 2$. If $|N_G(v) \cap X| = 2$ and $v$ is adjacent to two consecutive vertices of $X$ in $G$, say $N_G(v) \cap X = \{x_1, x_2\}$, then $G[x_1, x_2, v, x_4] \cong K_3 \sqcup K_1$, a contradiction. Thus if $|N_G(v) \cap X| = 2$, then \ref{u1} occurs. If $|N_G(v) \cap X| = 3$ and $v$ is adjacent to three consecutive vertices of $X$ in $G$, say $N_G(v) \cap X = \{x_1, x_2, x_3\}$, then $G[x_1, x_2, v, x_4] \cong K_3 \sqcup K_1$. Thus if $|N_G(v) \cap X| = 3$, then \ref{u2} occurs. 
If $|N_G(v) \cap X| = 4$, then $v$ is adjacent to four consecutive vertices of $X$ in $G$, say $N_G(v) \cap X = \{x_1, x_2, x_3, x_4\}$, then $G[x_2, x_3, x_5, v] \cong K_3 \sqcup K_1$, a contradiction. Thus, $|N_G(v) \cap X| = 5$, and \ref{u3} occurs. This completes the proof of the claim. \triqed\bigskip

Suppose for each $v \in V(G) \setminus X$, \ref{u1} does not occur. 
Let $H\coloneq G\setminus\{x_1,x_2,x_3,x_4\}$. Since $G$ is a minimal counterexample, $H$ has a small $K_{\chi(H)}$-model $M$. Note that every vertex of $H$ (including $x_5$) is adjacent to a vertex of $\{x_1,x_2\}$ and a vertex of $\{x_3,x_4\}$ in $G$, and $x_1 x_2, x_3 x_4, x_2 x_3 \in E(G)$. Therefore, $M'\coloneq M\cup \{\{x_1,x_2\},\{x_3,x_4\}\}$ is a small $K_{\chi(H)+2}$-model in $G$. Note that $\chi(G[x_1,x_2,x_3,x_4]) = 2$, so $\chi(H) \geq \chi(G) - 2$. Thus, $\had_2(G) \geq \had_2(H) + 2 \ge \chi(H) + 2 \geq \chi(G)$. This contradicts the assumption that $G$ is a minimal counterexample.

Therefore, there is a vertex $y \in V(G) \setminus X$ such that $N_G(y) \cap X = \{x_i, x_{i+2}\}$ for some $i \in \{1, \dots, 5\}$. Without loss of generality, suppose $N_G(y) \cap X = \{x_1, x_3\}$.
\begin{claim}
    \label{claim:twins}
    $N_G(x_2)=N_G(y)$.
\end{claim}
Suppose there is a vertex $w \in N_G(x_2) \setminus N_G(y)$.
Then $w \notin X \cup \{y\}$.
By \cref{claim:nghr5}, $w$ is adjacent to $x_4$ or $x_5$ in $G$. If $wx_4, wx_5 \in E(G)$, then $G[w, x_4, x_5, y] \cong K_3 \sqcup K_1$, a contradiction. Thus not both $x_4$ and $x_5$ are adjacent to $w$ in $G$. Without loss of generality, suppose that $wx_5 \in E(G)$ and $wx_4 \notin E(G)$. If $wx_1 \notin E(G)$, then $G[y, x_1, x_2, w, x_4] \cong P_4 \sqcup K_1$, a contradiction. Thus, $wx_1 \in E(G)$. However, then $\{x_1, x_2, x_5\} \subseteq N_G(w) \cap X$, but $N_G(w) \cap X \neq X$ since $x_4 \notin X$. Thus none of the cases in \cref{claim:nghr5} occurs, a contradiction. Hence, $N_G(x_2) \subseteq N_G(y)$. Swapping the roles of $x_2$ and $y$ (and considering the 5-cycle $\{x_1,y,x_3,x_4,x_5\}$ instead of $X$), we find also that $N_G(y) \subseteq N_G(x_2)$. This completes the proof of the claim.\triqed\bigskip

Let $H\coloneq G\setminus y$. Note that $\chi(H) = \chi(G)$, since a proper $\chi(H)$-coloring of $H$ can be extended to $G$ by coloring $y$ the same color as $x_2$ by \cref{claim:twins}. Since $G$ is a minimal counterexample, $H$ is not a counterexample, and we have $\had_2(G)\ge\had_2(H)\ge\chi(H)=\chi(G)$, a contradiction. Thus $G$ is $C_5$-free, and this completes the base case.

\bigskip

Now let us prove the inductive step. Suppose we have proven that $G$ is $\{K_3\sqcup K_1,P_4\sqcup K_1,\overline{C_5},\dots,\overline{C_{2k-1}}\}$-free for some $k\ge 3$; we wish to show that $G$ is also $\overline{C_{2k+1}}$-free. Suppose that $X=\{x_1,\dots,x_{2k+1}\}$ induces $\overline{C_{2k+1}}$ in $G$, i.e.\ $x_i$ is adjacent to $x_j$ if and only if $i-j\not\equiv \pm 1\pmod{2k+1}$. In what follows, we take indices modulo $2k+1$. 

\begin{claim}
    \label{claim:nghr_general}
For each $v\in V(G)\setminus X$, one of the following holds:
\begin{enumerate}[label=(\theclaim\alph*), itemindent=1em]
    \item \label{v1} $N_G(v)\cap X=X$, or
    \item \label{v2} $N_G(v)\cap X=X\setminus\{x_i,x_{i+1}\}$ for some $i\in\{1,\dots,2k+1\}$, or
    \item \label{v3} $N_G(v)\cap X=X\setminus\{x_i,x_{i+1},x_{i+2}\}$ for some $i\in\{1,\dots,2k+1\}$.
\end{enumerate}
\end{claim}
Let $v \in V(G) \setminus X$ and consider $\overline{G}$. If $x_i \in N_{\overline{G}}(v)$ and $x_{i-1}, x_{i+1} \notin N_{\overline{G}}(v)$ for some $i \in \{1, \dots, 2k+1\}$, then $G[v,x_{i-1},x_i,x_{i+1}] \cong K_3 \sqcup K_1$, a contradiction.
Thus $vx_{i-1}\in E(\overline{G})$ or $vx_{i+1}\in E(\overline{G})$. Without loss of generality, suppose $vx_{i+1}\in E(\overline{G})$. If $x_{i-1},x_{i+2} \in N_{\overline{G}}(v)$, then $G[x_{i-1},x_i,x_{i+1},x_{i+2},v] \cong P_4\sqcup K_1$, a contradiction. Thus $v$ is not adjacent in $\overline{G}$ to four or more consecutive vertices of $X$.

Therefore, $N_{\overline{G}}(v) \cap X$ can be partitioned into pairs and triplets of consecutive vertices of $X$. If $N_{\overline{G}}(v) \cap X = \varnothing$, then $N_G(v)\cap X=X$, and \ref{v1} occurs.
Suppose $v$ is adjacent in $\overline{G}$ to a triplet $\{x_i,x_{i+1},x_{i+2}\}$ and at least one other vertex of $X$, say $x_j$, with $j \notin \{i-1, i, i+1, i+2, i+3\}$. Then $G[x_i,x_{i+2},x_j,v]\cong K_3\sqcup K_1$, a contradiction. Hence, if $v$ is adjacent in $\overline{G}$ to a triplet of consecutive vertices of $X$, it is not adjacent to any other vertex of $X$ (in $\overline{G}$), and therefore such a $v$ satisfies \ref{v3}.

Thus, we may assume that $v$ is not adjacent to a triplet of consecutive vertices of $X$ in $\overline{G}$. If $v$ is adjacent in $\overline{G}$ to at least three distinct pairs $\{x_a,x_{a+1}\}$, $\{x_b,x_{b+1}\}$, $\{x_c,x_{c+1}\}$, then $G[x_a,x_b,x_c,v]\cong K_3\sqcup K_1$, a contradiction. If $v$ is adjacent in $\overline{G}$ to exactly two pairs $\{x_i,x_{i+1}\}$ and $\{x_j,x_{j+1}\}$, then consider the two sets
\begin{align*}
    Y_1&\coloneq\{x_{i+1},x_{i+2},\dots,x_{j-1},x_j,v\}, \\
    Y_2&\coloneq\{x_{j+1},x_{j+2},\dots,x_{i-1},x_i,v\}.
\end{align*}
Note that $\overline{G}[Y_1]$ and $\overline{G}[Y_2]$ are cycle graphs of orders between 4 and $2k-1$, inclusive. The sum of the lengths of these two cycles is precisely $2k+3$, which is odd, so one of these two cycles is an odd hole in $\overline{G}$ of length between 5 and $2k-1$, inclusive, and therefore an odd antihole in $G$ of the same length. However, this contradicts the inductive hypothesis. Thus if $v$ is adjacent in $\overline{G}$ to a pair of consecutive vertices of $X$, $v$ cannot be adjacent to a second pair, and so such a $v$ satisfies \ref{v2}. This completes the proof of the claim. \triqed\bigskip

Now consider $H\coloneq G\setminus\{x_1,\dots,x_6\}$. Since $G$ is a minimal counterexample, $H$ has a small $K_{\chi(H)}$-model $M$. Then, by \cref{claim:nghr_general}, every vertex of $H$ is adjacent to a vertex in each of the sets $\{x_1,x_4\}, \{x_2, x_5\}$ and $\{x_3,x_6\}$ in $G$. Therefore
\begin{equation*}
    M\cup\{\{x_1,x_4\},\{x_2,x_5\},\{x_3,x_6\}\}
\end{equation*}
is a small $K_{\chi(H)+3}$-model in $G$. Since $\chi(G[\{x_1,\dots,x_6\}])=3$, we have $\chi(G)\le \chi(H)+3$, and therefore $\had_2(G)\ge \chi(G)$, a contradiction.  This completes the inductive step and the proof.
\end{proof}

\section{Proof of Theorem~\ref{thm:strong_fork}}
\label{sec:fork}
In this section we prove \cref{thm:strong_fork}. We need the following structure theorem of $\{\text{fork}, \text{antifork}\}$-free graphs by Chudnovsky, Cook, and Seymour (definitions to follow):

\begin{theorem}[\cite{ccs}]\label{thm:structure}
Let $G$ be a $\{\text{fork}, \text{antifork}\}$-free graph. Then at least one of the following is true: 
\begin{enumerate}[label=(\roman*),leftmargin=5em]
    \item[(i), (ii)] $G$ or $\overline{G}$ is disconnected.
    \item[(iii), (iv)] $G$ or $\overline{G}$ has a pair of adjacent simplicial twins.
    \item[(v), (vi)] $G$ or $\overline{G}$ is candled.
    \item[(vii), (viii)] $G$ or $\overline{G}$ is the line graph of a simple triangle-free graph.
\end{enumerate}
\end{theorem}

Chudnovsky, Cook, and Seymour used \cref{thm:structure} to prove that $\chi(G)\le 2\omega(G)$ if $G$ is $\{\text{fork}, \text{antifork}\}$-free. We now explain the unfamiliar terms appearing in this structure theorem.

A vertex $v$ in a graph $G$ is \textit{simplicial} if its neighborhood $N_G(v)$ induces a clique. Two vertices $u,v \in V(G)$ are \textit{twins} if $N_G(u)\cap (V(G)\setminus \{u,v\}) = N_G(v)\cap (V(G)\setminus \{u,v\})$.

Two sets of vertices are \textit{complete} if every edge between them exists, and \textit{anticomplete} if there are no edges between them. A \textit{candelabrum} is a graph formed by $2k$ non-empty disjoint sets of vertices $Y_1,\dots,Y_k$ and $Z_1,\dots,Z_k$ such that:
\begin{itemize}
    \item Each $Y_i$ is a clique and each $Z_i$ is a stable set.
    \item The sets $Y_i$ are pairwise anticomplete and the sets $Z_i$ are pairwise complete.
    \item $Y_i$ is complete to $Z_j$ if $i=j$ and anticomplete to $Z_j$ if $i\ne j$.
\end{itemize}
We call the sets $Y_i$ the \textit{candle sets} and $Z_i$ the \textit{base sets}.

A graph $G$ is \textit{candled} if it contains a non-empty candelabrum with candle sets $Y_1,\dots,Y_k$ and base sets $Z_1,\dots,Z_k$ such that every vertex of $G$ not in the candelabrum is adjacent to every vertex in $Z_1,\dots,Z_k$ and non-adjacent to every vertex in $Y_1,\dots,Y_k$ (see \cref{fig:candled}). 
\begin{figure}[!htb]
    \centering
    \includegraphics[scale=0.9]{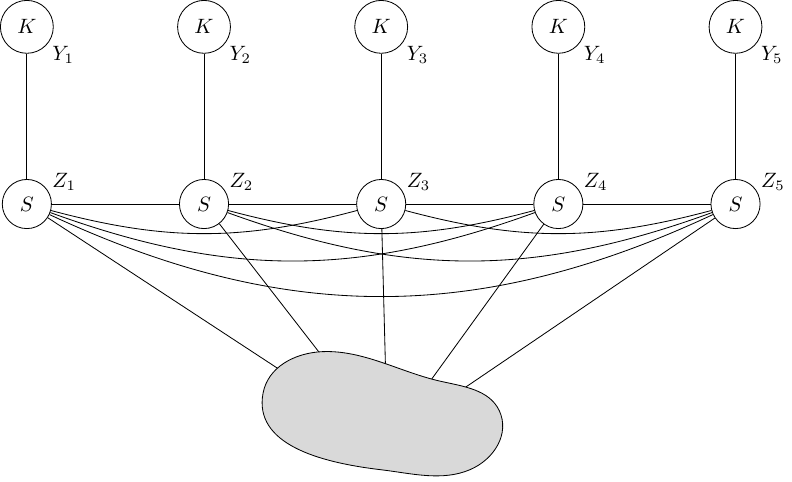}
    \caption{A candled graph with candelabrum $Y_1 \cup \dots \cup Y_5 \cup Z_1 \cup \dots \cup Z_5$. Circles and the grey blob represent sets of vertices. Lines between two sets indicate complete pairs, the absence of a line indicates anticomplete pairs, and $K$ and $S$ mark cliques and stable sets, respectively. The grey blob represents all of the vertices that are not in the candelabrum. This figure is modified from~\cite{ccs}.} 
    \label{fig:candled}
\end{figure}

We now show that a minimal $\{\text{fork}, \text{antifork}\}$-free counterexample to \cref{thm:strong_fork} cannot satisfy any of the eight outcomes in \cref{thm:structure}, thus proving \cref{thm:strong_fork}.

\begin{proposition}\label{prop:easy}
Let $G$ be a minimal counterexample to \cref{thm:strong_fork}. Then:
\begin{enumerate}[label=(\roman*)]
    \item~\label{o1} $G$ is not disconnected,
    \item~\label{o2} $\overline{G}$ is not disconnected,
    \item~\label{o3} $G$ has no simplicial vertex,
    \item~\label{o4} $G$ has no pair of non-adjacent twins,
    \item~\label{o5} $G$ is not candled,
    \item~\label{o6} $\overline{G}$ is not candled,
    \item~\label{o7} $G$ is not a line graph of a simple graph, and
    \item~\label{o8} $\overline{G}$ is not the line graph of a triangle-free (multi)graph.
\end{enumerate}
\end{proposition}
\begin{proof}
Note that $\chi(G)>\chi(H)$ for all proper induced subgraphs $H$ of $G$, as otherwise $\had_2^+(G)\ge\had_2^+(H)\ge\chi(H)=\chi(G)$.
\begin{enumerate}[label=(\roman*)]
    \item If $G=A\sqcup B$ for some non-empty graphs $A$ and $B$, then
    \[ \had_2^+(G)=\max(\had_2^+(A), \had_2^+(B))\ge \max(\chi(A), \chi(B))=\chi(G). \]
    \item Suppose $\overline{G}=\overline{A}\sqcup \overline{B}$ for some non-empty graphs $A$ and $B$. Then $\chi(G)=\chi(A) + \chi(B)$. Let $M=\{S_1,\dots,S_{\chi(A)}\}$ be a semi-small $K_{\chi(A)}$-model in $A$ with $\abs{S_i}\le 2$ for $i\ge 2$ and $N=\{T_1,\dots,T_{\chi(B)}\}$ be a semi-small $K_{\chi(B)}$-model in $B$ with $\abs{T_i}\le 2$ for $i\ge 2$. If $\abs{S_1}\le 1$ or $\abs{T_1}\le 1$ then $M\cup N$ is a semi-small $K_{\chi(G)}$-model. Otherwise, let $s_1,s_2\in S_1$ and $t_1,t_2\in T_1$ be distinct; then
    \[ \{S_2,\dots,S_{\chi(A)},T_2,\dots,T_{\chi(B)},\{s_1,t_1\},\{s_2,t_2\}\} \]
    is a small $K_{\chi(G)}$-model.
    \item Let $v \in V(G)$ be simplicial. Since $G$ is a minimal counterexample, $\chi(G\setminus v)=\chi(G)-1$, so $|N_G(v)| \geq \chi(G) - 1$. Since $v$ is simplicial, $N_G(v) \cup \{v\}$ induces a clique of size $\chi(G)$ in $G$, implying $\had_2^+(G)\ge\omega(G)=\chi(G)$.
    \item If $u,v \in V(G)$ are non-adjacent twins, then $\chi(G\setminus u)=\chi(G)$, since $u$ can be colored the color assigned to $v$ in a coloring of $G \setminus u$. Thus $G$ is not a minimal counterexample.
    \item Suppose $G$ is candled with base sets $Z_1,\dots,Z_k$ and candle sets $Y_1,\dots,Y_k$. If $|Z_i| \geq 2$ for some $i \in \{1, \dots, k\}$, then for distinct $u,v \in Z_i$, $u$ and $v$ are non-adjacent twins in $G$, contradicting~\ref{o4}. Thus $|Z_i| = 1$ for each $i \in \{1, \dots, k\}$. Therefore every vertex in each $Y_i$ is simplicial in $G$, contradicting~\ref{o3}.
    \item Suppose $\overline{G}$ is candled with base sets $Z_1,\dots,Z_k$ and candle sets $Y_1,\dots,Y_k$. If $|Y_i| \geq 2$ for some $i \in \{1, \dots, k\}$, then for distinct $u,v \in Y_i$, $u$ and $v$ are non-adjacent twins in $G$, contradicting~\ref{o4}.
    Thus, $|Y_i| = 1$ for each $i \in \{1, \dots, k\}$. Therefore every vertex in each $Z_i$ is simplicial in $G$ (the neighborhood of such a vertex $z$ in $G$ consists of $Z_i\setminus \{z\}$ plus the unique vertex in each $Y_j$ with $j\ne i$; these vertices are all pairwise adjacent in $G$), contradicting~\ref{o3}.
    \item Recall Vizing's theorem that $\chi(G)\le \omega(G)+1$ if $G$ is a line graph~\cite{vizing}. If $\omega(G) = \chi(G)$, then we are done, so assume $\omega(G) = \chi(G) - 1$. Let $C$ be a clique of $G$ of size $\omega(G)$, and let $A_1, \dots, A_k$ be non-empty connected components of $G \setminus C$. We will show that $k = 1$. If $k \geq 2$, then since $G$ is a minimal counterexample, $\chi(G[C \cup A_i]) < \chi(G)$ for each $i \in \{1, \dots, k\}$. By relabeling colors, we can combine the colorings of $G[C \cup A_1], \dots, G[C \cup A_k]$ to color $G$ using $\max(\chi(G[C\cup A_1]),\dots,\chi(G[C\cup A_k]))$ colors. Then $\chi(G) = \max(\chi(G[C\cup A_1]),\dots,\chi(G[C\cup A_k]))< \chi(G)$, a contradiction.
     
    Thus $k = 1$. If some vertex $c\in C$ was not adjacent to any vertex in $A_1$, then $c$ is simplicial, contradicting~\ref{o3}. Thus $\{\{c\}\}_{c\in C}\cup\{A_1\}$ is a semi-small $K_{\chi(G)}$-model. 
    \item This follows from \cref{cor:multigraphs}. \qedhere
\end{enumerate}
\end{proof}
Note that~\ref{o7} is the only case where a branch set of size larger than 2 is introduced.
One might wonder if it is possible to replace $\had_2^+(G)$ in \cref{thm:strong_fork} with $\had_2(G)$. This cannot be done, because the cycles $C_{2k+1}$ with $k\ge 3$ are line graphs of triangle-free graphs and have $2=\had_2(G)<\chi(G) = 3$, so branch sets of size larger than $2$ are indeed required in~\ref{o7}. In fact, there is a more general counterexample:
\begin{proposition}\label{prop:largegirth}
Let $H$ be a graph with girth at least 7, let $G$ be the line graph of $H$, and assume $\chi(G)=\omega(G)+1$. Then $\had_2(G)=\chi(G)-1$.
\end{proposition}
\begin{proof}
Note that connected subsets of $V(G)$ of size 1 or 2 correspond to the edge sets of 1- or 2-edge paths in $H$ respectively. Since $H$ has girth at least 7, three disjoint connected subsets of $V(G)$ of size at most 2 can only be pairwise adjacent if the corresponding edge sets in $H$ meet at a common vertex. Therefore $\had_2(G)\le\Delta(H)=\omega(G)$, where $\Delta(H)$ is the maximum degree in $H$, and so $\had_2(G)=\chi(G)-1$.
\end{proof}

An example of such $H$ is a $d$-regular graph of girth at least 7 with $\abs{V(H)}$ odd, which guarantees $\chi(G) = \omega(G) + 1$ (to see this, note that each color class of a coloring of $G$ corresponds to a matching in $H$ and therefore covers at most $(\abs{V(H)}-1)/2$ edges of $H$; there are $d\abs{V(H)}/2$ edges in $H$, so at least $d+1=\omega(G)+1$ color classes are required).

\section{Further results and open problems}
\label{sec:op}

One can easily improve the $\{\text{co-claw}, \text{co-gem}\}$-free result to a slightly more general class. Note that if $\Cc$ is a hereditary class such that $\had_2(G)\ge \chi(G)$ for all $G\in \Cc$, then the closure of $\Cc$ under disjoint unions is another class with the same property. The $\{\text{co-claw}, \text{co-gem}\}$-free graphs are not closed under disjoint unions. Let $\Cc'$ denote the closure of $\{\text{co-claw}, \text{co-gem}\}$-free graphs under disjoint unions, which corresponds to graphs whose connected components are $\{\text{co-claw}, \text{co-gem}\}$-free. We may obtain a forbidden induced subgraph characterization of $\Cc'$ as follows. Note that the co-claw with the isolated vertex removed is $K_3$, and the co-gem with the isolated vertex removed is the $4$-vertex path $P_4$. Let $G$ be a connected graph that contains the co-claw or co-gem as an induced subgraph, say on vertex set $Y$. Define $X \subseteq Y$ as follows: if $Y$ induces the co-claw, then $X$ induces $K_3$, and if $Y$ induces the co-gem, then $X$ induces $P_4$. Then there is at least one vertex $v \in V(G)$ anticomplete to $X$ (for instance, the unique vertex in $Y\setminus X$). Choose such a $v$ as close to $X$ as possible; then $v$ must be distance $2$ from $X$. Let $w$ be a neighbor of $v$ that is adjacent to a vertex of $X$. By considering the neighborhood of $w$ in $X$, we deduce that $G$ induces one of the graphs in \cref{fig:graphs3}. Hence:

\begin{proposition}\label{prop:connected}
Every connected component of $G$ is $\{\text{co-claw}, \text{co-gem}\}$-free if and only if $G$ is $S$-free, where $S$ is the set of graphs depicted in \cref{fig:graphs3}. Therefore if $G$ is $S$-free, then $\had_2(G) \geq \chi(G)$.
\end{proposition}
\begin{figure}[htb!]
\centering
    \includegraphics[]{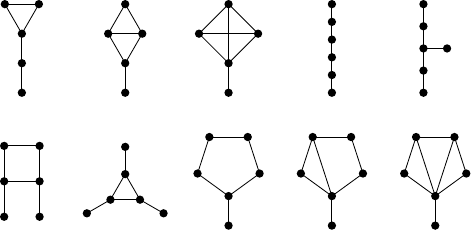}
    \caption{A forbidden induced subgraph characterization of $\Cc'$, the class of graphs whose connected components are $\{\text{co-claw},\text{co-gem}\}$-free.}
    \label{fig:graphs3}
\end{figure}

There are yet stronger results that can be easily obtained. For instance, let $S'$ be $S$ minus the third and tenth graphs in \cref{fig:graphs3}. Then one can prove that if $G$ is $S'$-free, then $\had_2(G)\ge \chi(G)$. This can be done by closing the class of $\{\text{co-claw},\text{co-gem}\}$ graphs under both disjoint unions and complete joins (disjoint unions in the complement) simultaneously. Here, we note that $\Cc'$ is not closed complete joins; a complete join of two graphs in $\Cc'$ might contain the third or tenth graph in \cref{fig:graphs3}.

\bigskip

We conclude with some further problems. \cref{cor:multigraphs} proves Hadwiger's Conjecture for complements of line graphs of triangle-free multigraphs. The first author has extended this result to all complements of line graphs of multigraphs. The proof uses some similar techniques but with substantial computer assistance; the details will appear in a future paper.
\cref{prop:largegirth} suggests the following broad problem:
\begin{problem}
\label{prob:sey}
For which hereditary classes $\Cc$ is it true that $\had_2(G) \geq \chi(G)$ for every graph $G \in \Cc$?
\end{problem}
Such a class $\Cc$ must at least avoid the construction in \cref{prop:largegirth}, so for instance $\Cc$ cannot contain odd cycles of length at least $7$. Perhaps the converse holds:
\begin{problem}
\label{prob:odd_cycle_free}
Does every $\{C_{2k+1}\}_{k\ge 3}$-free graph $G$ satisfy $\had_2(G) \geq \chi(G)$?
\end{problem}
An affirmative answer would be a substantial strengthening of the conjecture of Seymour that $\had_2(G)\ge \chi(G)$ if $G$ is $\overline{K_3}$-free~\cite{SeymourHC}. Here is an even stronger conjecture. For integer $m \geq 1$, let $\had_m(G)$ be the largest $t$ such that $G$ contains a $K_t$-model using branch sets of size at most $m$. Let $\Cc_m$ be the class of graphs $G$ such that $\had_m(H)\ge \chi(H)$ for all induced subgraphs $H$ of $G$. Then $\Cc_1$ is the class of perfect graphs, which we know is the class of $\{\text{odd hole}, \text{odd antihole}\}$-free graphs. It is easy to see that $\Cc_m$ is a subset of $\{C_{2k+1}\}_{k \ge 3m/2}$-free graphs, as $\had_m(C_\ell)$ is equal to 2 if $3m > \ell$ and 3 if $3m\le \ell$.

\begin{problem}
Is $\Cc_m$ the class of $\{C_{2k+1}\}_{k\ge 3m/2}$-free graphs for all $m\ge 2$?
\end{problem}

\section*{Acknowledgments}

Daniel Carter is supported by NSF grant DGE-2444107. Jung Hon Yip is supported by a Monash Graduate Scholarship. This work was partly done when the first author attended the 2025 Barbados Graph Theory Workshop at Bellairs Research Institute in Holetown, Barbados. We thank the organizers for providing an engaging work environment. We also thank David Wood for helpful comments on an early draft of this paper.

{
\fontsize{10pt}{11pt}
\selectfont
\bibliographystyle{bibstyle}
\bibliography{biblio}
}

\end{document}